\documentclass[12pt]{elsarticle}
\usepackage[utf8]{inputenc}
\usepackage{amsmath}
\usepackage{amssymb}
\usepackage{newtxtext}
\usepackage{newtxmath}
\usepackage{microtype}
\usepackage{xcolor}
\usepackage{graphicx}
\usepackage{enumerate}
\usepackage{bbold}
\makeatletter
\def\ps@pprintTitle{%
 \let\@oddhead\@empty
 \let\@evenhead\@empty
 \def\@oddfoot{\centerline{\thepage}}%
 \let\@evenfoot\@oddfoot}
\makeatother

\newtheorem{thm}{Theorem}[section]

\newtheorem{lemma}{Lemma}[section]

\newtheorem{cor}{Corollary}[section]

\usepackage{amssymb}

\begin{document}
\begin{frontmatter}

\title{From Sidelnikov-Welch bounds to projection constants}

 \author[label1]{Beata~Deregowska}
 
 \author[label4]{Barbara~Lewandowska \footnote{B.L. is partially supported by National Science Center (NCN) grant no. 2025/09/X/ST1/01604 }}

 \address[label1]{Institute of Mathematics\\
University of the National Education Commission, Podchorazych~2, Krakow, 30-084, Poland}
\address[label4]{Faculty of Mathematics and Computer Science\\
Jagiellonian University, Lojasiewicza~6, Krakow, 30-348, Poland}



\begin{abstract}
In the paper, we prove a recursive version of the weighted Sidelnikov-Welch
inequality for real and complex unit vectors. Unlike the classical
form, which gives a direct lower bound for a fixed even power sum, our
inequality relates two consecutive even power sums. Iteration yields
the usual weighted Sidelnikov-Welch bound.\\
~\\
We apply this estimate to maximal relative projection constants.
If $\mathbb{K}^m$ admits a maximal equiangular tight frame with
$M_{\mathbb K}$ vectors, then for every integer $k\geq1$,
$$
\lambda_{\mathbb K}(kM_{\mathbb K}-m,kM_{\mathbb K})
=
\lambda_{\mathbb K}(m)-\frac{2m}{kM_{\mathbb K}}+1.
$$
Moreover, the maximal value is realized by an equiangular tight frame
when $k=1$ and by a biangular tight frame when $k\geq2$.
\end{abstract}
\begin{keyword}

maximal absolute projection constant \sep 
maximal relative projection constant \sep 
equiangular tight frames \sep
biangular tight frames \sep
Sidelnikov-Welch bound\sep
(t,t)-designs



\MSC 41A65 \sep 41A44 \sep 46B20  \sep 42C15
\end{keyword}
\end{frontmatter}
\section{Introduction}

We begin by recalling the basic notions concerning projection constants. Let $X$ be a Banach space over $\mathbb{K},$ where $\mathbb{K}=\mathbb{R}$ or $\mathbb{K}=\mathbb{C}.$ Let $Y\subset X$ be a subspace.
By $\mathcal{P}(X,Y)$ denote the set of all linear and continuous projections from $X$ onto $Y$,
recalling that an operator $P \colon X \rightarrow Y$ is called a \textit{projection} onto $Y$ if $P|_Y={\rm Id}_Y.$ 
We define the \textit{relative projection constant} of subspace $Y$  and space $X$ by
\begin{equation*}
\lambda(Y,X) :=\inf\lbrace\|P\|:\;P\in\mathcal{P}(X,Y)\rbrace.
\end{equation*}
Notice that the set $\mathcal{P}(X,Y)$ can be empty (e.g. $\mathcal{P}(\ell_{\infty},c_{0})$). Then we will assume that $\lambda(Y,X)=\infty.$
Now we can define the \textit{absolute projection constant} of $Y$ by
\begin{equation}
\label{DefMAPC}
\lambda(Y) :=\sup\lbrace\lambda(Y,X):Y\subset X\rbrace.
\end{equation}
A natural question is how large the absolute projection constant can be among all $m$-dimensional Banach spaces. This leads to the notion of the maximal absolute projection constant,
\begin{equation*}
\lambda_{\mathbb{K}}(m) :=\sup \lbrace\lambda(Y):\; \dim(Y)=m \rbrace.
\end{equation*}
The problem of determining or estimating maximal absolute projection constants has attracted considerable attention in the theory of Banach spaces over the past decades.
By the Kadec--Snobar theorem (see \cite{KS}),
we have $\lambda_{\mathbb{K}}(m)\leq \sqrt{m}$. In 1960, B. Grünbaum conjectured that $\lambda_\mathbb{R}(2)=\frac{4}{3}$ (see \cite{G}), and in 2010, B. Chalmers and G. Lewicki proved it (see \cite{CL}). In 2019, G. Basso delivered an alternative proof of this conjecture (see \cite{B}). For many years it was widely believed that determining $\lambda_{\mathbb{K}}(m)$ for dimensions greater than two was beyond reach. However, connecting this problem with different structures of vectors led to some partial results and showed a possible way to solve it.

The connection with frame theory is based on the observation that the maximal absolute projection constant can be computed by considering finite-dimensional subspaces of $\ell_\infty$ (see, e.g., \cite[III.B.5]{W}).  Therefore, it can be defined as a supremum of \textit{maximal relative projection constants} for $N \ge m,$ given by 
\begin{equation*}
\lambda_{\mathbb{K}}(m,N):=\sup\lbrace \lambda(Y, \ell_\infty^{(N)}(\mathbb{K})):\; \dim(Y)=m \textrm{ and } Y\subset \ell_\infty^{(N)}(\mathbb{K})\rbrace.
\end{equation*}   
The crucial tool in our investigation is the following theorem, originally stated in \cite[Theorem 2.2]{CLe} and reproved in \cite[Theorem 1]{FS}.
\begin{thm}\label{lammbda}
For integers $N \ge m$, we have
\begin{align}\label{lambda1}
    \lambda_{\mathbb{K}}(m,N)&=\max\bigg\lbrace \sum_{i,j=1}^N t_it_j|U^* U|_{ij}:t\in\mathbb{R}_+^N,\;\|t\|=1,U\in \mathbb{K}^{m\times N},\; UU^*={\rm I}_m \bigg\rbrace 
\end{align}

\end{thm}
The matrix condition $UU^*={\rm I}_m$ in Theorem~\ref{lammbda} is precisely the Parseval frame condition. Thus, determining maximal relative projection constants can be reformulated as an optimization problem over Parseval frames. We begin by recalling the relevant terminology. A system of vectors $(u_1,\dots, u_N)$ in $\mathbb{K}^m$ is called a {\it tight frame} if there exists a constant $\alpha >0$ such that one of the following equivalent conditions holds:
\begin{itemize}
    \item  $\|x\|^2=\alpha\sum_{k=1}^{N}|\langle x, u_k \rangle|^2$  \; for all $x\in \mathbb{K}^m.$ 
     \item  $x=\alpha\sum_{k=1}^{N}\langle x, u_k \rangle u_k$  \; for all $x\in \mathbb{K}^m.$ 
    \item $UU^*=\frac{1}{\alpha}{\rm I_m}$, where $U$ is the matrix with columns $u_1,\dots, u_N.$ 
\end{itemize}   
If $\alpha= 1,$ then a tight frame is called {\it Parseval frame.}
\noindent The system $(u_1,\dots, u_N)$ of unit vectors in $\mathbb{K}^m$ is called an {\it equiangular tight frame} ETF$(m,N)$ if it is tight and
the value of $|\langle u_i, u_j\rangle|$ is constant over all $i\neq j.$
It is well known (see e.g. \cite[Theorem 5.7]{FR}) that if $(u_1,\ldots u_N )$ is an ETF$(m,N)$ then
\begin{equation}\label{Welch Bound}
|\langle u_i,u_j \rangle|=\sqrt{\frac{N-m}{m(N-1)}}
\qquad \textrm{ for all } i,j\in [1:N],\;i\neq j,
\end{equation}
and the constant $\alpha$ is also determined and is equal to $\frac{m}{N}.$
Notice that the existence of an equiangular tight frame consisting of $N$ unit vectors in $\mathbb{K}^m$ is equivalent to the existence of a Parseval frame such that
\begin{equation}\label{ETFangel}
    (U^*U)_{ii}=\frac{m}{N} \; \textrm{ and }\; |U^*U|_{ij}=\frac{m}{N}\sqrt{\frac{N-m}{m(N-1)}} \;\textrm{ for }\; i\neq j,
\end{equation}
where $i,j\in [1:N].$ Moreover, the quantity in \eqref{lambda1} for this Parseval frame with equal weights $t_i$ meets the upper bound for the maximal relative projection constant, given in \cite{KLL}. We present it in the form stated in \cite [ Theorem 5 ]{ FS}, where also an easier proof of this result was provided.
\begin{thm}\label{SF}

For integers $N\geq m$, the maximal relative projection constant $\lambda_{\mathbb{K}}(m,N)$ is upper bounded by
$$
\delta_{m,N} := \frac{m}{N} \left( 1 + \sqrt{\frac{(N-1)(N-m)}{m}} \right).
$$
Moreover, the following properties are equivalent:
\begin{enumerate}[i)]
\item There is an equiangular tight frame consisting of $N$ vectors in $\mathbb{K}^m,$
\item $\lambda_\mathbb{K}(m,N)=\tfrac{m}{N}\left(1 +\sqrt{\tfrac{(N-1)(N-m)}{m}} \right).$
\end{enumerate}
\end{thm}
When no ETF of $N$ vectors exists in $\mathbb K^m$, determining $\lambda_{\mathbb K}(m,N)$ is substantially more difficult. An important additional family of exact values arises in dimensions admitting ETFs with the largest possible number of vectors, usually referred to as maximal ETFs.
In 1994, H. K\"onig and N. Tomczak-Jaegermann stated the following estimation.
\begin{thm}[stated in \cite{KT}; proved in \cite{BB2}]\label{nMPC}
Let $ m>1.$ Then
\begin{enumerate}[i)]
    \item $\lambda_\mathbb{R}(m) \leq \frac{2}{m+1}\left(1+\frac{m-1}{2}\sqrt{m+2}\right)\,$ 
    \item $ \lambda_\mathbb{C}(m)\leq \frac{1}{m}\left(1+(m-1)\sqrt{m+1}\right).$ 
\end{enumerate}
\end{thm}
\noindent Observe that the upper bound in Theorem \ref{nMPC} is equal to $\delta_{m, \frac{m(m+1)}{2}} $ in the real case and $\delta_{m, m^2} $ in the complex case. The number of vectors in an ETF cannot exceed $\frac{m(m+1)}{2}$ in the real case and $m^2$ in the complex case (see, e.g., \cite[Theorem 5.10]{FR}). So if there exists the maximal  ETF in $\mathbb{K}^m$, the upper bound is realized. 

\begin{thm}[\cite{BB2}]\label{jakis}
Let $ m>1.$
\begin{enumerate}[i)]
    \item  If there exists a maximal ETF in $\mathbb{R}^m$ then $\lambda_\mathbb{R}(m) = \frac{2}{m+1}\left(1+\frac{m-1}{2}\sqrt{m+2}\right)\,$ 
    \item  If there exists a maximal ETF in $\mathbb{C}^m$ then $ \lambda_\mathbb{C}(m) = \frac{1}{m}\left(1+(m-1)\sqrt{m+1}\right).$ 
\end{enumerate}
\end{thm}
There are numerous examples of complex maximal ETFs (see, e.g., \cite{FM}). In fact, Zauner's conjecture asserts that such a frame exists in every complex dimension \cite{Z}. The situation is markedly different in the real case. A maximal real ETF can exist only for $m=2$, $m=3$, or
$
m=(2k+1)^2-2,
$ so the admissible dimensions are already very sparse. Maximal real ETFs are currently known only for $m=2,3,7,$ and $23$. Moreover, the next two arithmetically admissible dimensions, $m=47$ and $m=79$, have been ruled out (see \cite{BMV}).

Outside the above mentioned cases, exact computations of $\lambda(m, N)$ are rare. The main isolated examples are
$$
\lambda_{\mathbb R}(3,5)=\frac{5+4\sqrt2}{7}
\qquad\text{and}\qquad
\lambda_{\mathbb R}(4,6)=\frac53,
$$
and both required substantial and technically involved arguments (see,\cite{B, CLe}).

In this paper, we show that if $\mathbb{K}^m$ admits a maximal ETF consisting of $M_{\mathbb K}$ vectors, then, for every integer $k\geq 1$,

$$
\lambda_\mathbb{K}(kM_\mathbb{K}-m,kM_\mathbb{K})
=
\lambda_{\mathbb{K}}(m)-\frac{2m}{kM_\mathbb{K}}+1.
$$
The main tool in the proof is a recursive version of the weighted
Sidelnikov-Welch inequality \cite{We,Sid,SW}. To the best of our knowledge, this recursive
form has not appeared previously in the literature. We were led to it by
the projection-constant problem considered here. The usual
Sidelnikov-Welch inequality gives a direct lower bound for a sum involving
a fixed even power of the inner products, whereas our argument requires
more precise information relating two consecutive even powers.

More precisely, for unit vectors
$u_1,\ldots,u_N\in\mathbb K^m$ and real coefficients
$w_1,\ldots,w_N$, we prove that
\begin{equation}
\label{recursiveSWintro}
\sum_{i,j=1}^N
w_iw_j|\langle u_i,u_j\rangle|^{2t}
\geq
\frac{2t-1}{m+2t-2}
\sum_{i,j=1}^N
w_iw_j|\langle u_i,u_j\rangle|^{2t-2}
\end{equation}
in the real case, and
\begin{equation}
\label{recursiveSWintroC}
\sum_{i,j=1}^N
w_iw_j|\langle u_i,u_j\rangle|^{2t}
\geq
\frac{t}{m+t-1}
\sum_{i,j=1}^N
w_iw_j|\langle u_i,u_j\rangle|^{2t-2}
\end{equation}
in the complex case.

Iterating these inequalities recovers the usual weighted
Sidelnikov-Welch bound. The advantage of the recursive formulation is
therefore not an improvement of the final Welch constant itself, but the
fact that it retains the lower-order power sum instead of replacing it
immediately by a universal estimate. This distinction is essential for
our application. For $t=2$, the recursive inequality relates the sum of
fourth powers directly to the corresponding sum of squares, and the latter
can be combined with the Parseval-frame identities. The resulting estimate
is strong enough to cover a substantially larger range of parameters in
the complementary-dimension problem, leaving only a few low-dimensional
cases, all of which are already known.

\section{ Recursive version of the weighted Sidelnikov-Welch
inequality}
We begin by recalling the Fischer inner product and the notation needed below. Since our argument involves only real polynomials, we restrict throughout to the real setting and formulate the relevant properties of the Fischer inner product accordingly.
For a multi-index $\alpha\in\mathbb{N}_0^m$, we write
$
|\alpha| := \alpha_1 + \cdots + \alpha_m,
$
$
\alpha! := \alpha_1! \cdots \alpha_m!,
$
$
x^\alpha := x_1^{\alpha_1} \cdots x_m^{\alpha_m},
$
$
\partial^\alpha := \partial_1^{\alpha_1} \cdots \partial_m^{\alpha_m},
$
where $\partial_j := \partial/\partial x_j$.
To any polynomial $h = \sum_\alpha h_\alpha x^\alpha$, we associate the  differential operator $h(\partial) := \sum_\alpha h_\alpha \partial^\alpha$.
The \textit{Fischer inner product} of two polynomials $f = \sum_\alpha f_\alpha x^\alpha$ and $g = \sum_\alpha g_\alpha x^\alpha$ is defined by
\begin{equation}
\langle f, g \rangle_{\rm F} := \big( f(\partial) g \big)(0) = \sum_{\alpha \in \mathbb{N}_0^m} \alpha! \, f_\alpha g_\alpha.
\end{equation}
For $n \in \mathbb{N}_0$, we denote by $\mathcal{H}_n$ the space of homogeneous polynomials of degree $n$,
i.e., the space of polynomials of the form $f = \sum_{|\alpha| = n} f_\alpha x^\alpha.$  One basic property of the Fischer inner product is the adjointness between multiplication and differentiation.
Precisely, for all polynomials $f$, $g$, and $h$,
\begin{equation}
\label{muldif}
\langle h(\partial)f,  g \rangle_{\rm F}=\langle f, hg \rangle_{\rm F},
\end{equation}\label{FischerInner}
see e.g. \cite[Section~2]{AR}.
Applying the above to $h(x)=\|x\|^2$ and observing that
$h(\partial)=\Delta:=\sum_{j=1}^m\partial_j^2,$
we obtain
\begin{equation}
\label{lapmul}
\langle f,\|x\|^2g\rangle_{\rm F}
=
\langle\Delta f,g\rangle_{\rm F}.
\end{equation}
This relation will be used in the proof of the second identity in the
following lemma.

\begin{lemma}
\label{lemlaplacidentity}
For $n \ge 0$ and $g \in \mathcal{H}_n$,
\begin{align}
\label{eqlaplacian}
\Delta \big( \|x\|^2 g \big) & = \|x\|^2 \Delta g + (4n+2m) \, g,\\
\label{eqnlaplacian}
\big\| \|x\|^2 g \big\|_{\rm F}^2 & = \|\Delta g\|_{\rm F}^2 + (4n+2m) \, \|g\|_{\rm F}^2.
\end{align}
\end{lemma}
\noindent
{\sc Proof.}
Since $\|x\|^2=\sum_{i=1}^m x_i^2,$ for each $i\in[1:m]$ we have
$$
\partial_i\bigl(\|x\|^2g\bigr)
=
2x_i g+\|x\|^2\partial_i g.$$
Differentiating once more gives
$$
\partial_i^2\bigl(\|x\|^2g\bigr)
=
2g+4x_i\partial_i g+\|x\|^2\partial_i^2 g.$$
Summing over $i$ yields
\begin{equation}
\label{someequality}
\Delta\bigl(\|x\|^2g\bigr)
=
2mg
+
4\sum_{i=1}^m x_i\partial_i g
+
\|x\|^2\Delta g.
\end{equation}
Since $g$ is homogeneous of degree $n$, we may write
$g(x)=\sum_{|\alpha|=n} g_\alpha x^\alpha.$
Hence \begin{align*}
\sum_{i=1}^m x_i\partial_i g
&=
\sum_{i=1}^m
x_i\partial_i
\Big(
\sum_{|\alpha|=n}g_\alpha x^\alpha
\Big)=
\sum_{|\alpha|=n}
g_\alpha
\sum_{i=1}^m
\alpha_i x^\alpha\\
&=
\sum_{|\alpha|=n}
|\alpha|\,g_\alpha x^\alpha
=
n\sum_{|\alpha|=n}g_\alpha x^\alpha
=
ng.
\end{align*}
Substituting this into \eqref{someequality} proves \eqref{eqlaplacian}. For the second identity, \eqref{lapmul} and
\eqref{eqlaplacian} give
\begin{align*}
\bigl\|\|x\|^2g\bigr\|_{\rm F}^2
&=
\bigl\langle\|x\|^2g,\|x\|^2g,\rangle_{\rm F}=
\bigl\langle\Delta(\|x\|^2g),g\bigr\rangle_{\rm F}\\
&=
\bigl\langle\|x\|^2\Delta g,g\bigr\rangle_{\rm F}
+
(4n+2m)\|g\|_{\rm F}^2\\
&=
\|\Delta g\|_{\rm F}^2
+
(4n+2m)\|g\|_{\rm F}^2,
\end{align*}
where in the last line we used \eqref{lapmul} once more.
This proves \eqref{eqnlaplacian}.\\
~\\
We next use a simple consequence of the definition of the Fischer inner
product. For $u\in\mathbb R^m$ and $n\geq 0$, the multinomial theorem gives
\begin{equation}
\label{multi}
\langle u,x\rangle^n
=
\sum_{|\alpha|=n}
\frac{n!}{\alpha!}\,u^\alpha x^\alpha.
\end{equation}
Hence, if
$$
f(x)=\sum_{|\alpha|=n}f_\alpha x^\alpha\in\mathcal H_n,
$$
then
\begin{align}
\label{scalerepr}
\big\langle f,\langle u,x\rangle^n\big\rangle_{\rm F}
&=
\sum_{|\alpha|=n}
\alpha!\,f_\alpha
\frac{n!}{\alpha!}\,u^\alpha=
n!\sum_{|\alpha|=n}f_\alpha u^\alpha
=
n!\,f(u).
\end{align}
In particular, taking $f(x)=\langle v,x\rangle^n$, we obtain
\begin{equation}
\label{eqpowers}
\big\langle
\langle u,x\rangle^n,
\langle v,x\rangle^n
\big\rangle_{\rm F}
=
n!\,\langle u,v\rangle^n,
\qquad
u,v\in\mathbb R^m.
\end{equation}
Identity \eqref{eqpowers} connects the Fischer inner product with the
weighted sums of powers of inner products occurring in the Sidelnikov-Welch-type
bound. In view of \eqref{eqpowers}, let
$u_1,\ldots,u_N\in\mathbb R^m$ be unit vectors and
$w_1,\ldots,w_N\in\mathbb R$. For $t\in\mathbb{N}_0$, define
\begin{equation}
\label{Pt}
P_t(x)
:=
\sum_{i=1}^N w_i\langle u_i,x\rangle^{2t}
\in\mathcal H_{2t}.
\end{equation}
The polynomials $P_t$ satisfy several simple identities involving the
Fischer inner product and the Laplacian. We collect them in the following
lemma.
\begin{lemma}
\label{lempolynomials}
Let $P_t$ be defined by \eqref{Pt}. For $t\in \mathbb{N}_0$,
\begin{align}
\|P_t\|_{\rm F}^2
&=(2t)!\sum_{i,j=1}^N w_iw_j|\langle u_i,u_j\rangle|^{2t},
\label{Ptnorm}
\end{align}
and
\begin{align}
\big\langle P_{t+1},\|x\|^2P_t\big\rangle_{\rm F}
&=(2t+2)(2t+1)\|P_t\|_{\rm F}^2
=(2t+2)!\sum_{i,j=1}^N
w_iw_j|\langle u_i,u_j\rangle|^{2t}.
\label{Ptinner}
\end{align}
For every $t\geq1$, one also has
\begin{align}
\Delta P_t
&=2t(2t-1)P_{t-1},
\label{Ptlap}\\
\big\|\|x\|^2P_t\big\|_{\rm F}^2
&=\big(2t(2t-1)\big)^2\|P_{t-1}\|_{\rm F}^2
+2(4t+m)\|P_t\|_{\rm F}^2.
\label{Pttsq}
\end{align}
\end{lemma}

\noindent
\noindent
{\sc Proof.}
By bilinearity and \eqref{eqpowers},
\begin{align*}
\|P_t\|_{\rm F}^2
&=\sum_{i,j=1}^N w_iw_j
\big\langle \langle u_i,x\rangle^{2t},
\langle u_j,x\rangle^{2t}\big\rangle_{\rm F}
=(2t)!\sum_{i,j=1}^N
w_iw_j|\langle u_i,u_j\rangle|^{2t},
\end{align*}
which proves \eqref{Ptnorm}.
For $u\in\mathbb R^m$,
$$
\Delta\langle u,x\rangle^{2t}
=2t(2t-1)\|u\|^2\langle u,x\rangle^{2t-2}.
$$
Since $\|u_i\|=1$, it follows immediately that
$$
\Delta P_t=2t(2t-1)P_{t-1},
$$
which is \eqref{Ptlap}. Consequently, by \eqref{lapmul},
\begin{align*}
\big\langle P_{t+1},\|x\|^2P_t\big\rangle_{\rm F}
&=\big\langle\Delta P_{t+1},P_t\big\rangle_{\rm F}=(2t+2)(2t+1)\|P_t\|_{\rm F}^2.
\end{align*}
Together with \eqref{Ptnorm}, this gives \eqref{Ptinner}. Finally, applying \eqref{eqnlaplacian} to
$P_t\in\mathcal H_{2t}$ and using \eqref{Ptlap}, we obtain
\begin{align*}
\big\|\|x\|^2P_t\big\|_{\rm F}^2
&=\|\Delta P_t\|_{\rm F}^2
  +2(4t+m)\|P_t\|_{\rm F}^2
=\big(2t(2t-1)\big)^2\|P_{t-1}\|_{\rm F}^2
  +2(4t+m)\|P_t\|_{\rm F}^2,
\end{align*}
which proves \eqref{Pttsq}.

\begin{thm}
\label{ddineq}
Let $m,N\geq1$, let $u_1,\ldots,u_N\in\mathbb K^m$ be unit vectors,
and let $w_1,\ldots,w_N\in\mathbb R$. Then, for every integer
$t\geq1$,
\begin{equation}
\label{inqR}
\sum_{i,j=1}^N w_iw_j|\langle u_i,u_j\rangle|^{2t}
\geq
\frac{2t-1}{m+2t-2}
\sum_{i,j=1}^N w_iw_j|\langle u_i,u_j\rangle|^{2t-2},
\qquad \mathbb K=\mathbb R,
\end{equation}
and
\begin{equation}
\label{inqC}
\sum_{i,j=1}^N w_iw_j|\langle u_i,u_j\rangle|^{2t}
\geq
\frac{t}{m+t-1}
\sum_{i,j=1}^N w_iw_j|\langle u_i,u_j\rangle|^{2t-2},
\qquad \mathbb K=\mathbb C.
\end{equation}
Moreover, for a fixed $t\geq1$, equality in the corresponding
inequality holds if and only if
\begin{equation}
\label{equalitycondition}
\sum_{i=1}^N w_i|\langle u_i,x\rangle|^{2t}
=
c_t(m,\mathbb K)
\Big(\sum_{i=1}^N w_i\Big)\|x\|^{2t},
\qquad x\in\mathbb K^m,
\end{equation}
where
\begin{equation}
c_t(m,\mathbb R)
=
\frac{1\cdot3\cdots(2t-1)}
     {m(m+2)\cdots(m+2t-2)},
\qquad
c_t(m,\mathbb C)
=
\frac{1}{\binom{m+t-1}{t}}.
\end{equation}
In this case, for every integer $1\leq s\leq t$,
\begin{equation}
\label{lowerconditions}
\sum_{i=1}^N w_i|\langle u_i,x\rangle|^{2s}
=
c_s(m,\mathbb K)
\Big(\sum_{i=1}^N w_i\Big)\|x\|^{2s},
\qquad x\in\mathbb K^m.
\end{equation}
\end{thm}
{\sc Proof.}
We first consider the real case. For the given unit vectors $u_1,\ldots,u_N$ and real coefficients
$w_1,\ldots,w_N$, let $P_t$ be defined as in \eqref{Pt}. By \eqref{Ptnorm}, inequality
\eqref{inqR} is equivalent to
\begin{equation}
\label{Ptineq}
\|P_t\|_{\rm F}^2
\geq
\frac{2t(2t-1)^2}{m+2t-2}\,
\|P_{t-1}\|_{\rm F}^2.
\end{equation}
If $P_{t-1}=0$, the claim is immediate. We may therefore assume that
$P_{t-1}\neq 0$.
Now we proceed by induction on $t.$
For the initial step, \eqref{Ptinner} and \eqref{eqnlaplacian} give
$$
\big\langle P_1,\|x\|^2P_0\big\rangle_{\rm F}
=2\|P_0\|_{\rm F}^2,
\qquad
\big\|\|x\|^2P_0\big\|_{\rm F}^2
=2m\|P_0\|_{\rm F}^2.
$$
Hence, by the Cauchy-Schwarz inequality,
$$
4\|P_0\|_{\rm F}^4
\leq
2m\|P_1\|_{\rm F}^2\|P_0\|_{\rm F}^2,
$$
and therefore
$$
\|P_1\|_{\rm F}^2
\geq
\frac{2}{m}\|P_0\|_{\rm F}^2.
$$
Let us now assume that the induction hypothesis holds up to $t,$ and let us show that it holds for $t+1,$ too. Then \eqref{Pttsq} and the assumption give
\begin{align*}
\big\|\|x\|^2P_{t}\big\|_{\rm F}^2 
&\leq4t^2(2t-1)^2\|P_{t-1}\|_{\rm F}^2+2(4t+m)\|P_{t}\|_{\rm F}^2\\
&\leq
\big(2t(m+2t-2)+2(4t+m)\big)\|P_t\|_{\rm F}^2
=
2(t+1)(m+2t)\|P_t\|_{\rm F}^2.
\end{align*}
Combining \eqref{Ptinner} with the Cauchy-Schwarz inequality and the
estimate above, we get
\begin{align*}
\big((2t+2)(2t+1)\big)^2\|P_t\|_{\rm F}^4
&\leq
\|P_{t+1}\|_{\rm F}^2
\big\|\|x\|^2P_t\big\|_{\rm F}^2
\leq
(2t+2)(m+2t)
\|P_{t+1}\|_{\rm F}^2\|P_t\|_{\rm F}^2.
\end{align*}
Hence
$$
\|P_{t+1}\|_{\rm F}^2
\geq
\frac{2(t+1)(2t+1)^2}{m+2t}\,
\|P_t\|_{\rm F}^2.
$$
This shows that the induction hypothesis holds for $t+1$ and concludes the inductive proof.

We now characterize the case of equality. Suppose first that equality
holds in \eqref{Ptineq}. If $P_t=0$, then necessarily $P_{t-1}=0$,
and successive applications of \eqref{Ptlap} give
$P_{t-2}=\cdots=P_0=0$. Thus the required identities hold trivially. Assume therefore that $P_t\neq0$. It follows from the proof of \eqref{Ptineq} that equality for $t$
forces equality for every $s\in[1:t].$
In particular, equality must hold in each application of the
Cauchy-Schwarz inequality. Thus $P_s$ and
$\|\cdot\|^2P_{s-1}$ are linearly dependent for every $s\in[1:t]$.
Repeatedly applying this observation shows that there exist constants $C_{s}\in \mathbb{R}$ such that
\begin{equation*}
    P_s(x)=C_{s}\|x\|^{2s}P_0(x)=C_{s}\|x\|^{2s}\sum_{i=1}^{N}w_i.
\end{equation*}
It remains to determine the coefficients $C_t$. Substituting the above
form of $P_s$ into \eqref{Ptlap}, we obtain
$$
2s(2s-1)C_{s-1}\|x\|^{2s-2}\sum_{i=1}^{N}w_i
=
\Delta P_s(x)
=
2s(m+2s-2)C_s\|x\|^{2s-2}\sum_{i=1}^{N}w_i,
$$
and hence
$$
C_s=
\frac{2s-1}{m+2s-2}\,C_{s-1}.
$$
Since $C_0=1$, iteration gives
$
C_t
=
c_t(m,\mathbb R),
$
as required.\\
Conversely, suppose that
\begin{equation}
\label{radialPt}
P_t(x)
=
c_t(m,\mathbb R)
\Big(\sum_{i=1}^N w_i\Big)\|x\|^{2t}.
\end{equation}
Then
$$
\|x\|^2\Delta P_t
=
2t(m+2t-2)P_t.
$$
Using \eqref{lapmul} and \eqref{Ptlap}, we obtain
\begin{align*}
\big(2t(2t-1)\big)^2\|P_{t-1}\|_{\rm F}^2
&=
\|\Delta P_t\|_{\rm F}^2
=
\big\langle P_t,\|x\|^2\Delta P_t\big\rangle_{\rm F}
=
2t(m+2t-2)\|P_t\|_{\rm F}^2.
\end{align*}
After rearranging,
$$
\|P_t\|_{\rm F}^2
=
\frac{2t(2t-1)^2}{m+2t-2}\,
\|P_{t-1}\|_{\rm F}^2,
$$
which is equality in \eqref{Ptineq}, and hence in \eqref{inqR}.

We now turn to the complex case. Throughout this part,
$\langle\cdot,\cdot\rangle_{\mathbb C}$ and
$\langle\cdot,\cdot\rangle_{\mathbb R}$ denote the standard inner
products on $\mathbb C^m$ and $\mathbb R^{2m}$, respectively.  Let
$\rho:\mathbb C^m\to\mathbb R^{2m}$ be the realification map
$
\rho(z):=({\rm Re}z,\operatorname{Im}z).
$
Then $\rho$ is an isometry and
\begin{equation}
\label{realification}
\big\langle\rho(z),\rho(y)\big\rangle_{\mathbb R}
=
{\rm Re}\langle z,y\rangle_{\mathbb C},
\qquad z,y\in\mathbb C^m.
\end{equation}
Fix $t\geq1$. In the sequel, we will need the following elementary identity. For every
$a\in\mathbb C$ and every  $k\in[0:t]$,
\begin{equation}
\label{rotationidentity}
\frac{1}{2t+1}\sum_{s=0}^{2t}
\Bigl(
{\rm Re}\bigl(e^{2\pi\mathrm{i}s/(2t+1)}a\bigr)
\Bigr)^{2k}
=
4^{-k}\binom{2k}{k}|a|^{2k}.
\end{equation}
Indeed, using
$$
{\rm Re}\bigl(e^{2\pi\mathrm{i}s/(2t+1)}a\bigr)
=
\frac{1}{2}
\left(
e^{2\pi\mathrm{i}s/(2t+1)}a
+
e^{-2\pi\mathrm{i}s/(2t+1)}\bar a
\right),
$$
the binomial theorem gives
\begin{align*}
\Bigl(
{\rm Re}\bigl(e^{2\pi\mathrm{i}s/(2t+1)}a\bigr)
\Bigr)^{2k}
&=
4^{-k}\sum_{\ell=0}^{2k}
\binom{2k}{\ell}
a^\ell\bar a^{\,2k-\ell}
e^{2\pi\mathrm{i}s(2\ell-2k)/(2t+1)}.
\end{align*}
Therefore,
\begin{align*}
\frac{1}{2t+1}\sum_{s=0}^{2t}
\Bigl(
{\rm Re}\bigl(e^{2\pi\mathrm{i}s/(2t+1)}a\bigr)
\Bigr)^{2k}\!\!=
4^{-k}\sum_{\ell=0}^{2k}
\binom{2k}{\ell}a^\ell\bar a^{\,2k-\ell}
\left(
\frac{1}{2t+1}\sum_{s=0}^{2t}
e^{2\pi\mathrm{i}s(2\ell-2k)/(2t+1)}
\right).
\end{align*}
The sum in parentheses equals $1$ when $\ell=k$ and vanishes otherwise.
Therefore,
\begin{align*}
\frac{1}{2t+1}\sum_{s=0}^{2t}
\Bigl(
{\rm Re}\bigl(e^{2\pi\mathrm{i}s/(2t+1)}a\bigr)
\Bigr)^{2k}
&=
4^{-k}\binom{2k}{k}a^k\bar a^k
=
4^{-k}\binom{2k}{k}|a|^{2k},
\end{align*}
which proves \eqref{rotationidentity}.
For $i\in[1:N]$ and $s\in[0:2t]$, define
$$
\widetilde{u}_{is}
:=
\rho\bigl(e^{2\pi\mathrm{i}s/(2t+1)}u_i\bigr)
\in\mathbb R^{2m},
\qquad
\widetilde{w}_{is}
:=
\frac{w_i}{2t+1}.
$$
Since $\rho$ is an isometry, each $\widetilde{u}_{is}$ is a unit vector
in $\mathbb R^{2m}$. Moreover, by \eqref{realification},
\begin{align*}
\big\langle\widetilde{u}_{ir},\widetilde{u}_{js}\big\rangle_{\mathbb R}
&=
{\rm Re}
\big\langle
e^{2\pi\mathrm{i}r/(2t+1)}u_i,
e^{2\pi\mathrm{i}s/(2t+1)}u_j
\big\rangle_{\mathbb C}=
{\rm Re}\Bigl(
e^{2\pi\mathrm{i}(r-s)/(2t+1)}
\langle u_i,u_j\rangle_{\mathbb C}
\Bigr).
\end{align*}
We now relate the sums for the real configuration to those for the
original complex vectors. Fix $0\leq k\leq t$. Since the inner product
$\langle\widetilde{u}_{ir},\widetilde{u}_{js}\rangle_{\mathbb R}$
is determined by the value of $r-s$ modulo $2t+1$, we group the
terms according to this difference. For each
$h\in[0:2t]$ there are exactly $2t+1$ pairs $(r,s)$ satisfying
$r-s\equiv h\pmod{2t+1}$, one for each choice of $s$. Hence,
using \eqref{rotationidentity}, we have
\begin{align}\label{sumidentity}
\sum_{i,j=1}^N\sum_{r,s=0}^{2t}
\widetilde{w}_{ir}\widetilde{w}_{js}
\big|\langle\widetilde{u}_{ir},\widetilde{u}_{js}\rangle_{\mathbb R}\big|^{2k}&=
\sum_{i,j=1}^N w_iw_j\frac{1}{2t+1}
\sum_{h=0}^{2t}
\Bigl(
{\rm Re}\bigl(
e^{2\pi\mathrm{i}h/(2t+1)}
\langle u_i,u_j\rangle_{\mathbb C}
\bigr)
\Bigr)^{2k}\notag \\
&=
4^{-k}\binom{2k}{k}
\sum_{i,j=1}^N
w_iw_j
\big|\langle u_i,u_j\rangle_{\mathbb C}\big|^{2k}.
\end{align}
Applying the real inequality \eqref{inqR} in $\mathbb R^{2m}$ to the
unit vectors $\widetilde{u}_{is}$ with coefficients
$\widetilde{w}_{is}$, we obtain
\begin{align*}
\sum_{i,j=1}^N\sum_{r,s=0}^{2t}
\widetilde{w}_{ir}\widetilde{w}_{js}
\big|\langle\widetilde{u}_{ir},\widetilde{u}_{js}\rangle_{\mathbb R}\big|^{2t}
\geq
\frac{2t-1}{2m+2t-2}
\sum_{i,j=1}^N\sum_{r,s=0}^{2t}
\widetilde{w}_{ir}\widetilde{w}_{js}
\big|\langle\widetilde{u}_{ir},\widetilde{u}_{js}\rangle_{\mathbb R}\big|^{2t-2}.
\end{align*}
Using the identity \eqref{sumidentity} with $k=t$ on the left-hand side and
$k=t-1$ on the right-hand side, the above reads
\begin{multline*}
4^{-t}\binom{2t}{t}
\sum_{i,j=1}^N w_iw_j
\big|\langle u_i,u_j\rangle_{\mathbb C}\big|^{2t}
\geq
\frac{(2t-1)4^{-(t-1)}}{2m+2t-2}
\binom{2t-2}{t-1}
\sum_{i,j=1}^N w_iw_j
\big|\langle u_i,u_j\rangle_{\mathbb C}\big|^{2t-2}.
\end{multline*}
After simplifying the constants, we have
$$
\sum_{i,j=1}^N w_iw_j
\big|\langle u_i,u_j\rangle_{\mathbb C}\big|^{2t}
\geq
\frac{t}{m+t-1}
\sum_{i,j=1}^N w_iw_j
\big|\langle u_i,u_j\rangle_{\mathbb C}\big|^{2t-2},
$$
as required. We now characterize the case of equality.  Since \eqref{inqC} was
obtained from \eqref{inqR} for the real configuration
$\{\widetilde{u}_{is},\widetilde{w}_{is}\}$ using only identities,
equality in \eqref{inqC} holds if and only if equality holds in
\eqref{inqR} for this configuration.
By the real case, this is equivalent to
\begin{equation}
\label{complequal}
\sum_{i=1}^N\sum_{r=0}^{2t}
\widetilde{w}_{ir}
\big|\langle\widetilde{u}_{ir},y\rangle_{\mathbb R}\big|^{2k}
=
c_k(2m,\mathbb R)
\left(\sum_{i=1}^N w_i\right)\|y\|^{2k},
\end{equation}
for every $1\leq k\leq t$ and every $y\in\mathbb R^{2m}$.
For $x\in\mathbb C^m$, set $y=\rho(x)$. Then
\eqref{realification}, \eqref{rotationidentity}, and the definitions of
$\widetilde{u}_{ir}$ and $\widetilde{w}_{ir}$ give
\begin{align*}
\sum_{i=1}^N\sum_{r=0}^{2t}
\widetilde{w}_{ir}
\big|\langle\widetilde{u}_{ir},\rho(x)\rangle_{\mathbb R}\big|^{2k}
&=
\sum_{i=1}^N
w_i\frac{1}{2t+1}
\sum_{r=0}^{2t}
\Bigl(
{\rm Re}\bigl(
e^{2\pi{\rm i}r/(2t+1)}
\langle u_i,x\rangle_{\mathbb C}
\bigr)
\Bigr)^{2k}\\
&=
4^{-k}\binom{2k}{k}
\sum_{i=1}^N
w_i|\langle u_i,x\rangle_{\mathbb C}|^{2k}.
\end{align*}
Since $\rho$ is an isometry from $\mathbb C^m$ onto $\mathbb R^{2m}$,
condition \eqref{complequal} is therefore equivalent to
\begin{equation}
\label{complexequalitycondition}
\sum_{i=1}^N
w_i|\langle u_i,x\rangle_{\mathbb C}|^{2k}
=
\frac{c_k(2m,\mathbb R)}
     {4^{-k}\binom{2k}{k}}
\Big(\sum_{i=1}^N w_i\Big)\|x\|^{2k},
\end{equation}
for every $1\leq k\leq t$ and every $x\in\mathbb C^m$.
After simple calculations, we get
\begin{equation}
\label{complexlowerconditions}
\sum_{i=1}^N
w_i|\langle u_i,x\rangle_{\mathbb C}|^{2k}
=
c_k(m,\mathbb C)
\Big(\sum_{i=1}^N w_i\Big)\|x\|^{2k},
\qquad 1\leq k\leq t.
\end{equation}
For $k=t$ we recover the equality condition in \eqref{inqC}, while
$k<t$ gives the corresponding identities for the lower powers.

Iterating Theorem~\ref{ddineq} yields the following corollary. For nonnegative weights summing to one, it reduces to the weighted Sidelnikov-Welch bound.

\begin{cor}
Let $u_1,\ldots,u_N\in\mathbb K^m$ be unit vectors and let
$w_1,\ldots,w_N\in\mathbb R$. Then, for every integer $t\geq1$,
\begin{equation}
\sum_{i,j=1}^N
w_iw_j|\langle u_i,u_j\rangle|^{2t}
\geq
c_t(m,\mathbb K)
\left(\sum_{i=1}^N w_i\right)^2,
\end{equation}
where
\begin{equation}
    c_t(m,\mathbb{R})=\frac{1\cdot 3 \cdot 5 \cdots (2t-1)}{m(m+2)\cdots(m+2(t-1))},\;\;
     c_t(m,\mathbb{C})=\frac{1}{\binom{m+t-1}{t} }
\end{equation}
\end{cor}
\section{Application to maximal relative projection constants}
We now turn to the proof of our main result. We first derive from Theorem~\ref{ddineq} an estimate for Parseval frames, which will then be applied to the complementary-dimension problem.
\begin{lemma}
\label{BukhCox}

Let $1<m < N.$  Let $U\in \mathbb{K}^{m\times N}$ such that $UU^*=I_m$ and let $t\in \mathbb{R}_+^{N}$ such that $\|t\|_2\leq1.$ Then 
\begin{equation}\label{LemmaInequalityR}
\sum_{i,j=1}^{N}t_it_j|\langle u_i, u_j\rangle|\leq\frac{2+(m-1)\sqrt{m+2}}{2(m+1)^2}\left(m+2+\left(\sum_{i=1}^N t_i\|u_i\| \right)^2\right), \;\; \textit{for } \; \mathbb{K}=\mathbb{R}
\end{equation}
and 
\begin{equation}\label{LemmaInequalityC}
\sum_{i,j=1}^{N}t_it_j|\langle u_i, u_j\rangle|\leq\frac{\sqrt{m+1}}{2} +\frac{1+(\tfrac{1}{2}m-1)\sqrt{m+1}}{m^2}\left(\sum_{i=1}^N t_i\|u_i\| \right)^2, \;\; \textit{for } \; \mathbb{K}=\mathbb{C}
\end{equation}
where  $u_i$  denotes the $i$-th column of the matrix $U.$
\end{lemma}
{\sc Proof.} Let $U$ and $t$ be as in the assumptions. Observe that we can assume that all columns of matrix $U$ are nonzero. If not, then we can  form a new matrix by keeping only the nonzero columns and define a new vector $t$ by keeping the coordinates corresponding to those columns, the sums in \eqref{LemmaInequalityR} and \eqref{LemmaInequalityC} remain unchanged, while the resulting matrix and the reduced vector $t$ still satisfy all the assumptions of the lemma. Now define two constants depending on the field $\mathbb{K}$
\begin{equation}\label{const}
    \varphi=\begin{cases}
        \frac{1}{\sqrt{m+2}}& for \;\; \mathbb{K}=\mathbb{R}\\
        \frac{1}{\sqrt{m+1}}& for \;\; \mathbb{K}=\mathbb{C}
    \end{cases}
\;\;\;and\;\;\; 
 c=\begin{cases}
        \frac{3}{m+2}& for \;\; \mathbb{K}=\mathbb{R}\\
        \frac{2}{m+1}& for \;\; \mathbb{K}=\mathbb{C}
    \end{cases}
\end{equation}
Combining the Cauchy-Schwarz inequality and Theorem \ref{ddineq}  we get
\begin{align*}
\sum_{i,j=1}^N &t_it_j\frac{(|\langle u_i,u_j\rangle|-\varphi
\|u_i\|\|u_j\|)^2}{\|u_i\|\|u_j\|}
=\sum_{i,j=1}^N t_it_j\frac{(|\langle u_i,u_j\rangle|^2-\varphi^2\|u_i\|^2\|u_j\|^2)^2}{\|u_i\|\|u_j\|(|\langle u_i,u_j\rangle|+\varphi\|u_i\|\|u_j\|)^2}\\
&\geq \sum_{i,j=1}^Nt_it_j \frac{(|\langle u_i,u_j\rangle|^2-\varphi^2\|u_i\|^2\|u_j\|^2)^2}{(1+\varphi)^2\|u_i\|^3\|u_j\|^3} = \left(\sum_{i,j=1}^N \left|\left\langle \frac{u_i}{\|u_i\|}, \frac{u_j}{\|u_j\|} \right\rangle\right|^4\|t_iu_i\|\|t_ju_j\|\right.\\
&\left.-2\varphi^2 \sum_{i,j=1}^Nt_it_j\frac{|\langle u_i, u_j \rangle|^2}{\|u_i\|\|u_j\|}+\varphi^4\sum_{i,j=1}^Nt_it_j\|u_i\|\|u_j\|\right)(1+\varphi)^{-2}  \\
&\geq\left((c-2\varphi^2) \sum_{i,j=1}^Nt_it_j\frac{|\langle u_i, u_j \rangle|^2}{\|u_i\|\|u_j\|}+\varphi^4 \big(\sum_{i=1}^Nt_i\|u_i\|\big)^2\right)(1+\varphi)^{-2}
\end{align*}
Squaring the addends of the left-sided sum and rearranging the latter inequality gives 
\begin{align}\label{inq2}\notag
    2\varphi\sum_{i,j=1}^{N}t_it_j|U^*U|_{ij} &\leq \Big(1+\frac{2\varphi^2-c}{(1+\varphi)^2}\Big)\sum_{i,j=1}^Nt_it_j\frac{|\langle u_i, u_j \rangle|^2}{\|u_i\|\|u_j\|} \\
    &+
    \big(\sum_{i=1}^Nt_i\|u_i\|\big)^2\varphi^2\left(1-\frac{\varphi^2}{(1+\varphi)^2}\right)
\end{align}
Now observe that using the Cauchy-Schwarz inequality and the tightness of the vectors $u_1,\dots,u_N$, we have
\begin{align*}
    \sum_{i,j=1}^{N}t_it_j\frac{|\langle u_i, u_j\rangle|^2}{\|u_i\|\|u_j\|}&\leq \sqrt{\sum_{i,j=1}^{N}\frac{t_i^2}{\|u_i\|^2}|\langle u_i, u_j\rangle|^2}\sqrt{\sum_{i,j=1}^{N}\frac{t_j^2}{\|u_j\|^2}|\langle u_i, u_j\rangle|^2}\\
    &=\sum_{i=1}^N\frac{t_i^2}{\|u_i\|^2}\sum_{j=1}^{N}|\langle u_i,u_j\rangle|^2
    =\sum_{i=1}^Nt_i^2\leq1.
\end{align*}
One can easily see that $1+\frac{2\varphi^2-c}{(1+\varphi)^2} >0$ in both cases ($\mathbb{K}=\mathbb{R}$ and $\mathbb{K}=\mathbb{C}$). So using the latter and substituting the corresponding values of $\varphi$ and $c$ from \eqref{const} yields the desired upper bound.\\

\noindent Observe that Lemma~\ref{BukhCox}, together with the elementary estimate
\begin{align}
\label{mn}
\left(\sum_{i=1}^{N}t_i\|u_i\|\right)^2
\leq
\sum_{i=1}^N t_i^2\sum_{i=1}^{N}\|u_i\|^2
=
{\rm tr}(U^*U)
=
{\rm tr}(UU^*)
=
m,
\end{align}
immediately yields Theorem~\ref{nMPC}.\\
We now use Lemma~\ref{BukhCox}  in the complementary-dimension setting to prove our main result.

\begin{thm}
Assume that for $m\geq 2$, there exists a maximal ETF in $\mathbb{K}^m.$ Let $M_\mathbb{R}=\frac{m(m+1)}{2}$ and $M_\mathbb{C}=m^2.$ Then for $k\geq1$ 
$$\lambda_{\mathbb{K}}(kM_{\mathbb{K}}-m,kM_{\mathbb{K}})= \lambda_{\mathbb{K}}(m)-\frac{2m}{kM_{\mathbb{K}}}+1.$$ Moreover, this value is attained by an equiangular tight frame when
$k=1$, and by a biangular tight frame when $k\geq 2$.
\end{thm}
{\sc Proof.} Let $d:=k M_\mathbb{K}-m$ and $N=kM_\mathbb{K}.$ Fix matrix $V\in\mathbb{K}^{d\times N},$ such that $VV^{*}=I_{d},$ and  unit vector $t\in \mathbb{R}_{+}^{N}.$ Since the rows of $V$ form an orthonormal system in $\mathbb{K}^{N},$ we may extend them to an orthonormal basis of $\mathbb{K}^{N}.$ The remaining $m$ basis vectors can be taken as the rows of a matrix $U\in\mathbb{K}^{m\times N}.$ Then

\begin{equation}
    V^*V=I_N-U^*U \;\;\; and \;\;\; UU^*=I_m
\end{equation}
Since $U^*U$ is an orthogonal projection $\|u_i\|^2=|\langle e_i,U^*Ue_i\rangle|\leq 1,$ where $u_i$ denotes i-th column of matrix $U.$ 
Observe also that using the Cauchy-Schwarz inequality we have
\begin{equation*}
  \left(\sum_{i=1}^{N}t_i\|u_i\|\right)^2=\langle[1,\dots,1], [t_i\|u_1\|,\dots, t_{N}\|u_N\|]\rangle^2  \leq N\sum_{i=1}^{N}t_i^2\|u_i\|^2 
\end{equation*}
and consequently
\begin{align*}
    \sum_{i,j=1}^{N} t_it_j|\langle v_i,v_j\rangle|&= \sum_{\substack{i,j=1 \\ i\neq j}}^{N} t_it_j|\langle u_i,u_j\rangle|+\sum_{i=1}^{N}t_i^2(1-\|u_i\|^2)\\
    &= 1 +\sum_{i,j=1}^{N} t_it_j|\langle u_i,u_j\rangle|-2\sum_{i=1}^{N}t_i^2\|u_i\|^2\\
    &\leq 1 +\sum_{i,j=1}^{N} t_it_j|\langle u_i,u_j\rangle|-\frac{2}{N}\left(\sum_{i=1}^{N}t_i\|u_i\|\right)^2
\end{align*}
Now, using inequality \eqref{mn} and Lemma \ref{BukhCox} we obtain for the real case
\begin{align*}
    \sum_{i,j=1}^{N} t_it_j|\langle v_i,v_j\rangle|&\leq 1+\frac{(2+(m-1)\sqrt{m+2})(m+2)} {2(m+1)^2}\\
    &+\left(\frac{2+(m-1)\sqrt{m+2}} {2(m+1)^2}-\frac{2}{N} \right)\left(\sum_{i=1}^N t_i\|u_i\| \right)^2\\
    &\leq 1+ \frac{2}{m+1}\left(1+\frac{m-1}{2}\sqrt{m+2}\right)-\frac{2m}{kM_{\mathbb{R}}}=1+\lambda_{\mathbb{R}}(m)-\frac{2m}{kM_{\mathbb{R}}}
\end{align*}
Since the above estimate holds for every $V$ and $t$, we obtain
$\lambda_\mathbb{R}(d,N) \leq 1+\lambda_{\mathbb{R}}(m)-\frac{2m}{kM_{\mathbb{R}}}$
In the above calculations, we are using the fact that 
$$
\frac{2+(m-1)\sqrt{m+2}} {2(m+1)^2}-\frac{2}{N} =\frac{1}{(m+1)^2}\left(1+\frac{m-1}{2}\sqrt{m+2}-\frac{4}{k}(1+\frac{1}{m})\right)\geq0
$$
Among the dimensions $m$ for which a maximal real ETF exists, the
above expression is nonnegative, with only the following exceptions:
\begin{itemize}
    \item $m=2,$ $k=1.$ Then
    \begin{equation*}
        \lambda_\mathbb{R}(d,N)=\lambda_\mathbb{R}(1,3)=1=\lambda_{\mathbb{R}}(m)-\frac{2m}{kM_\mathbb{R}}+1
    \end{equation*}
    \item $m=2,$ $k=2.$ Then by result of Basso \cite{B}
    \begin{equation*}
        \lambda_\mathbb{R}(d,N)=\lambda_\mathbb{R}(4,6)=\frac{5}{3}=\lambda_{\mathbb{R}}(m)-\frac{2m}{kM_\mathbb{R}}+1.
    \end{equation*}
     \item $m=3,$ $k=1.$ Then 
    \begin{equation*}
        \lambda_\mathbb{R}(d,N)=\lambda_\mathbb{R}(3,6)=\lambda_{\mathbb{R}}(m)=\lambda_{\mathbb{R}}(m)-\frac{2m}{kM_\mathbb{R}}+1.
    \end{equation*}
\end{itemize}
Analogously for complex case we get  
\begin{align*}
   \sum_{i,j=1}^{N} t_it_j|\langle v_i,v_j\rangle|&\leq 1+\frac{\sqrt{m+1}} {2}
    +\left(\frac{1+(\tfrac{1}{2}m-1)\sqrt{m+1}} {m^2}-\frac{2}{N} \right)\left(\sum_{i=1}^N t_i\|u_i\| \right)^2\\
    &\leq 1+ \frac{1}{m}\left(1+(m-1)\sqrt{m+1}\right)-\frac{2m}{kM_{\mathbb{C}}}=1+\lambda_{\mathbb{C}}(m)-\frac{2m}{kM_{\mathbb{C}}}
\end{align*}
Since the above estimate holds for every $V$ and $t$, we obtain
$\lambda_\mathbb{R}(d,N) \leq 1+\lambda_{\mathbb{C}}(m)-\frac{2m}{kM_{\mathbb{R}}}$
This time we are using the fact that 
\begin{equation*}
    \frac{1+(\tfrac{1}{2}m-1)\sqrt{m+1}} {m^2}-\frac{2}{N} =\frac{1}{m^2}\left(1+(\tfrac{1}{2}m-1)\sqrt{m+1}-\frac{2}{k}\right)
\geq 0\end{equation*}
It is always true except for $m=2$ and $k=1,$ but then
\begin{equation*}
    \lambda_\mathbb{C}(d,N)=\lambda_\mathbb{C}(2,4)=\lambda_\mathbb{C}(m)=\lambda_\mathbb{C}(m)-\frac{2m}{kM_\mathbb{C}}+1.
\end{equation*}
We next construct a biangular tight frame that attains the above upper bound.
 Let $U$ be a matrix for which the value $\lambda_\mathbb{K}(m)$ is attained (generated by a maximal ETF in $\mathbb{K}^m$). Then the matrix $U_k:=\sqrt{\frac{1}{k}}\underbrace{[U|\ldots |U]}_{\text{k times}}$ satisfies $U_kU_k^\ast ={\rm I}_m.$ Let $V$ be a matrix whose rows form an orthonormal completion of the rows of $U_k.$ Since
$${\rm I}_{N}=U_k^*U_k+V^*V,$$
 we obtain
 \begin{equation*}
 \begin{aligned}
   V^*V&=  {\rm I}_{N}-\frac{1}{k}  \left[\begin{array}{c|c|c}
    U^*U  & \dots &  U^*U \\
      \hline
         \vdots& \ddots & \vdots\\
      \hline   
      U^*U  & \dots &  U^*U
    \end{array}\right]
    \end{aligned}
    \end{equation*}
 Since the columns of $U$ have equal norms, $V$ is an equal-norm tight frame. Moreover, for distinct columns of $V$, the identity above shows that their absolute inner products take  the values $|\langle u_i, u_j\rangle|$ for $k=1$ and the values
$$
\frac{1}{k}\|u_i\|^2
\qquad\text{or}\qquad
\frac{1}{k}|\langle u_i,u_j\rangle|,\quad i\neq j \quad for \quad k\geq2.
$$
Since $U$ is generated by an ETF, these give exactly one value for $k=1$ and two values for
$k\geq2$. Consequently

$$
\frac{1}{N}\sum_{i,j=1}^{N}|\langle v_i,v_j\rangle|=\frac{k}{N}\sum_{i,j=1}^{M_\mathbb{K}} |\langle u_i,u_j\rangle|  +1-\frac{2}{N}\sum_{i=1}^{M_\mathbb{K}}\|u_i\|^2=\lambda_\mathbb{K}(m)+1-\frac{2m}{kM_\mathbb{K}},
$$
as required.

\end{document}